\documentclass[12pt]{amsart}
\usepackage{amsmath, amsthm, amscd, amsfonts, amssymb, graphicx, color,float,pgf,tikz}
\usepackage{amssymb,fontenc}
\usepackage{latexsym,wasysym,mathrsfs}
\usepackage{hyperref}
\usetikzlibrary{arrows}
\usepackage[square,numbers,sort&compress]{natbib}

\makeatletter \oddsidemargin.9375in \evensidemargin \oddsidemargin
\newtheorem{theorem}{Theorem}[section]
\newtheorem{lemma}[theorem]{Lemma}
\newtheorem{proposition}[theorem]{Proposition}

\theoremstyle{definition}
\newtheorem{definition}[theorem]{Definition}
\newtheorem{example}[theorem]{Example}

\numberwithin{equation}{section}

\newcommand{\be}{\begin{equation}}
\newcommand{\ee}{\end{equation}}
\newcommand{\bee}{\begin{example}}
\newcommand{\eee}{\end{example}}

\numberwithin{equation}{section}
\usepackage{etoolbox}
\makeatletter
\patchcmd{\@settitle}{\uppercasenonmath\@title}{}{}{}
\patchcmd{\@setauthors}{\MakeUppercase}{}{}{}
\makeatother

\allowdisplaybreaks
\begin{document}

\title[ Controlled $\ast$-K-operator frame for $End_\mathcal{A}^\ast (\mathcal{H})$]{Controlled $\ast$-K-operator frame for $End_\mathcal{A}^\ast (\mathcal{H})$}

\author[H. Labrigui, M. Rossafi, A. Touri and Nadia Assila]{Hatim Labrigui$^{1}$, Mohamed Rossafi$^{*2}$, Abdeslam Touri$^{1}$ and Nadia Assila$^{1}$}
\address{$^{1}$Department of Mathematics, Faculty of Sciences, Ibn Tofail University, Kenitra, Morocco}
\email{hlabrigui75@gmail.com; touri.abdo68@gmail.com; nadia.assila@uit.ac.ma}
\address{$^{2}$LASMA Laboratory Department of Mathematics, Faculty of Sciences Dhar El Mahraz, University Sidi Mohamed Ben Abdellah, Fes, Morocco}
\email{\textcolor[rgb]{0.00,0.00,0.84}{rossafimohamed@gmail.com; mohamed.rossafi@usmba.ac.ma}}


\subjclass[2010]{42C15, 41A58}


\keywords{$\ast$-K-Operator Frame, Controlled $\ast$-K-operator frame,    $C^{\ast}$-algebra, Hilbert $C^{\ast}$-modules.\\
\indent $^{*}$ Corresponding author}
\maketitle

\begin{abstract}
	Frame Theory has a great revolution for recent years. This theory has been extended from Hilbert spaces to Hilbert  $C^{\ast}$-modules. In this paper, we introduce the concept of Controlled $\ast$-$K$-operator frame for the space $End_{\mathcal{A}}^{\ast}(\mathcal{H})$ of all adjointable operators on a Hilbert $\mathcal{A}$-module $\mathcal{H}$ and we establish some results.

\end{abstract}
\maketitle
\vspace{0.1in}

\section{\textbf{Introduction and preliminaries}}
In 1946, Gabor \cite{a} introduced a method for reconstructing functions (signals) using a family of elementary functions. Later in 1952, Duffin and Schaeffer \cite{Duf} presented a similar tool in the context of nonharmonic Fourier series and this is the starting point of frame theory. After some decades, Daubechies, Grossmann and Meyer \cite{13} in 1986 announced formally the definition of frame in the abstract Hilbert spaces. 

 Controlled frames in Hilbert spaces have been introduced by P. Balazs \cite{01} to improve the numerical efficiency of iterative algorithms for inverting the frame operator.
 
 In this paper, we define and study the notion of controlled $\ast$-$K$-operator frames for $End_{\mathcal{A}}^{\ast}(\mathcal{H})$, to consider the relation between $\ast$-$K$-operator frames and controlled $\ast$-$K$-operator frames in a given Hilbert $C^{\ast}$-module and we establish some results.
 
In the following, we briefly recall the definitions and basic properties of $C^{\ast}$-algebra, Hilbert $\mathcal{A}$-modules. Our references for $C^{\ast}$-algebras as \cite{{Dav},{Con}}. For a $C^{\ast}$-algebra $\mathcal{A}$ if $a\in\mathcal{A}$ is positive we write $a\geq 0$ and $\mathcal{A}^{+}$ denotes the set of all positive elements of $\mathcal{A}$.
\begin{definition}\cite{Pas}	
	Let $ \mathcal{A} $ be a unital $C^{\ast}$-algebra and $\mathcal{H}$ be a left $ \mathcal{A} $-module, such that the linear structures of $\mathcal{A}$ and $ \mathcal{H} $ are compatible. $\mathcal{H}$ is a pre-Hilbert $\mathcal{A}$-module if $\mathcal{H}$ is equipped with an $\mathcal{A}$-valued inner product $\langle.,.\rangle_{\mathcal{A}} :\mathcal{H}\times\mathcal{H}\rightarrow\mathcal{A}$, such that is sesquilinear, positive definite and respects the module action. In the other words,
	\begin{itemize}
		\item [(i)] $ \langle x,x\rangle_{\mathcal{A}}\geq0 $ for all $ x\in\mathcal{H} $ and $ \langle x,x\rangle_{\mathcal{A}}=0$ if and only if $x=0$.
		\item [(ii)] $\langle ax+y,z\rangle_{\mathcal{A}}=a\langle x,z\rangle_{\mathcal{A}}+\langle y,z\rangle_{\mathcal{A}}$ for all $a\in\mathcal{A}$ and $x,y,z\in\mathcal{H}$.
		\item[(iii)] $ \langle x,y\rangle_{\mathcal{A}}=\langle y,x\rangle_{\mathcal{A}}^{\ast} $ for all $x,y\in\mathcal{H}$.
	\end{itemize}	 
	For $x\in\mathcal{H}, $ we define $||x||=||\langle x,x\rangle_{\mathcal{A}}||^{\frac{1}{2}}$. If $\mathcal{H}$ is complete with $||.||$, it is called a Hilbert $\mathcal{A}$-module or a Hilbert $C^{\ast}$-module over $\mathcal{A}$. For every $a$ in $C^{\ast}$-algebra $\mathcal{A}$, we have $|a|=(a^{\ast}a)^{\frac{1}{2}}$ and the $\mathcal{A}$-valued norm on $\mathcal{H}$ is defined by $|x|=\langle x, x\rangle_{\mathcal{A}}^{\frac{1}{2}}$ for $x\in\mathcal{H}$.
\end{definition} 	
	Let $\mathcal{H}$ and $\mathcal{K}$ be two Hilbert $\mathcal{A}$-modules, A map $T:\mathcal{H}\rightarrow\mathcal{K}$ is said to be adjointable if there exists a map $T^{\ast}:\mathcal{K}\rightarrow\mathcal{H}$ such that $\langle Tx,y\rangle_{\mathcal{A}}=\langle x,T^{\ast}y\rangle_{\mathcal{A}}$ for all $x\in\mathcal{H}$ and $y\in\mathcal{K}$.
	
We reserve the notation $End_{\mathcal{A}}^{\ast}(\mathcal{H},\mathcal{K})$ for the set of all adjointable operators from $\mathcal{H}$ to $\mathcal{K}$ and $End_{\mathcal{A}}^{\ast}(\mathcal{H},\mathcal{H})$ is abbreviated to $End_{\mathcal{A}}^{\ast}(\mathcal{H})$. Also, $GL^{+}(\mathcal{H})$ is the set of all positive bounded linear invertible operators on $\mathcal{H}$ with bounded inverse.
 
The following lemmas will be used to prove our mains results.
\begin{lemma} \label{111} \cite{Ara}.
	Let $\mathcal{H}$ and $\mathcal{K}$ be two Hilbert $\mathcal{A}$-modules and $T\in End_{\mathcal{A}}^{\ast}(\mathcal{H},\mathcal{K})$. Then the following statements are equivalent:
	\begin{itemize}
		\item [(i)] $T$ is surjective.
		\item [(ii)] $T^{\ast}$ is bounded below with respect to norm, i.e.: there is $m>0$ such that $\|T^{\ast}x\|\geq m\|x\|$, $x\in\mathcal{K}$.
		\item [(iii)] $T^{\ast}$ is bounded below with respect to the inner product, i.e.: there is $m'>0$ such that, 
		\begin{equation*}
	\langle T^{\ast}x,T^{\ast}x\rangle_\mathcal{A}\geq m'\langle x,x\rangle_\mathcal{A}, \quad  x\in\mathcal{K}.
	\end{equation*}
	\end{itemize}
	\end{lemma}
\begin{lemma} \label{1} \cite{Pas}.
	Let $\mathcal{H}$ be an Hilbert $\mathcal{A}$-module. If $T\in End_{\mathcal{A}}^{\ast}(\mathcal{H})$, then $$\langle Tx,Tx\rangle_{\mathcal{A}}\leq\|T\|^{2}\langle x,x\rangle_{\mathcal{A}}, \qquad x\in\mathcal{H}.$$
\end{lemma}

For the following theorem, $R(T)$ denote the range of the operator $T$.
\begin{theorem}\label{3} \cite{Do} 
	Let E, F and G be Hilbert ${\mathcal{A}}$-modules over a $C^{\ast}$-algebra $\mathcal{A}$. Let $ T\in End_{\mathcal{A}}^{\ast}(E,F) $ and  $  T'\in End_{\mathcal{A}}^{\ast}(G,F) $ with 
    $\overline{({R(T^{\ast})})}$ is orthogonally complemented. Then the following statements are equivalent:
	\begin{itemize}
		\item [(1)] $T'(T')^{\ast} \leq \lambda TT^{\ast}$ for some $\lambda >0$.
		\item  [(2)]	There exists $\mu >0 $ such that $\|(T')^{\ast}x\|\leq \mu \|T^{\ast}x\|$ for all $x\in F$.
		\item  [(3)] There exists $ D\in End_{\mathcal{A}}^{\ast}(G,E)$ such that $T'=TD$,
		that is the equation $TX=T'$ has a solution.
		\item  [(4)]  $R(T')\subseteq R(T)$.
		
	\end{itemize}
\end{theorem}	

\section{\textbf{Controlled $\ast$-$K$-operator frame for $End_\mathcal{A}^\ast (\mathcal{H})$}}
We begin this section with the following definition.

\begin{definition}
	Let $K \in End_{\mathcal{A}}^{\ast}(\mathcal{H})$ and  $C,C^{'} \in GL^{+}(\mathcal{H})$. A family of adjointable operators $\{T_{i}\}_{i\in I}$ on a Hilbert $\mathcal{A}$-module $\mathcal{H}$ over a unital $C^{\ast}$-algebra is said to be a $(C,C^{'})$-controlled $\ast$-K-operator frame for $End_{\mathcal{A}}^{\ast}(\mathcal{H})$, if there exist two strictly non zero elements $A$ and $B$ in $\mathcal{A}$ such that
	\begin{equation}\label{1}
	A\langle K^{\ast}x,K^{\ast}x\rangle_{\mathcal{A}} A^{\ast} \leq\sum_{i\in I}\langle T_{i}Cx,T_{i}C^{'}x\rangle_{\mathcal{A}}\leq B\langle x,x\rangle_{\mathcal{A}}B^{\ast},  x\in\mathcal{H}.
	\end{equation}
	The elements $A$ and $B$ are called respectively lower and upper bounds of the $(C,C^{'})$-controlled $\ast$-K-operator frame.\\
	If $A\langle K^{\ast}x,K^{\ast}x\rangle_{\mathcal{A}} A^{\ast} =\sum_{i\in I}\langle T_{i}Cx,T_{i}C^{'}x\rangle_{\mathcal{A}}$, the $(C,C^{'})$-controlled $\ast$-K-operator frame  is called $A$-tight.\\
	If $A = 1_{\mathcal{A}}$, it is called a normalized tight $(C,C^{'})$-controlled $\ast$-K-operator frame or a Parseval $(C,C^{'})$-controlled $\ast$-K-operator frame.
\end{definition}
\begin{example}
	Let $\mathcal{H}=l_{2}(\mathbb{C})=\bigg\{\{a_{n}\}_{n\in \mathbb{N}} \subset \mathbb{C} ; \; \sum_{n\in\mathbb{N}}|a_{n}|^{2}<\infty\bigg\}$ be a Hilbert space with respect the inner product,
	\begin{equation*}
	\langle \{a_{n}\}_{n\in \mathbb{N}},\{b_{n}\}_{n\in \mathbb{N}}\rangle =\sum_{n\in \mathbb{N}}a_{n}\bar{b_{n}},
	\end{equation*}
	equiped with the norm,
	\begin{equation*}
	\|\{a_{n}\}_{n\in \mathbb{N}}\|_{l_{2}(\mathbb{C})}=(\sum_{n\in \mathbb{N}}|a_{n}|^{2})^{\frac{1}{2}}.
	\end{equation*}
	We consider the $\mathbb{C}^{\ast}$-algebra $\mathcal{A}=\bigg\{\{a_{n}\}_{n\in \mathbb{N}}\subset \mathbb{C}; \; \underset{n \in \mathbb{N}}{max}|a_{n}|<\infty \bigg\}$, equiped with the involution,
	\begin{equation*}
	\{a_{n}\}_{n\in \mathbb{N}} \longrightarrow \{a_{n}\}^{\ast}_{n\in \mathbb{N}}=\{\bar{a_{n}}\}_{n\in \mathbb{N}}\quad for\; all \quad \{a_{n}\}_{n\in \mathbb{N}}\subset \mathcal{A}
	\end{equation*}
	Let the map,
	\begin{align*}
	\mathcal{H} \times \mathcal{H} &\longrightarrow \mathcal{A}\\
	(\{a_{n}\}_{n\in \mathbb{N}},\{b_{n}\}_{n\in \mathbb{N}})&\longrightarrow \langle \{a_{n}\}_{n\in \mathbb{N}},\{b_{n}\}_{n\in \mathbb{N}}\rangle_{\mathcal{A}}= \bigg\{\frac{a_{n}\overline{b_{n}}}{n}\bigg\}_{n\in \mathbb{N}}
	\end{align*}
	Wich is an inner product and equiped with it, $\mathcal{H}$ is a Hilbert $\mathcal{A}$-modules.\\
	Let
	\begin{align*}
	\Lambda_{k}:\mathcal{H}&\longrightarrow \mathcal{H}\\
	\{a_{n}\}_{n\in \mathbb{N}}&\longrightarrow \bigg\{\frac{a_{n}\delta_{n}^{2k+1}}{\sqrt{2k+1}}\bigg\}_{n\in \mathbb{N}},
	\end{align*}
	where $\delta_{n}^{2k+1}$ is the Kronecker symbol.\\
	$\Lambda_{k}$ is defined as follow,
	\begin{equation*}
	\Lambda_{k}(\{a_{n}\}_{n\in \mathbb{N}})=(0,0,...,\frac{a_{2k+1}}{\sqrt{2k+1}},0,...)\quad For\; all \quad \{a_{n}\}_{n\in \mathbb{N}} \subset \mathcal{H}.\\
	\end{equation*}
	
	Hence the family $\{\Lambda_{k}\}_{k\in \mathbb{N}}$ is bounded and linear.\\
	For $\alpha , \beta > 0$, we define,
	\begin{align*}
	C:\mathcal{H} &\longrightarrow \mathcal{H}\\
	\{a_{n}\}_{n\in \mathbb{N}}&\longrightarrow \{\alpha a_{n}\}_{n\in \mathbb{N}},
	\end{align*}
	and 
	\begin{align*}
	C^{'}:\mathcal{H} &\longrightarrow \mathcal{H}\\
	\{a_{n}\}_{n\in \mathbb{N}}&\longrightarrow \{\beta a_{n}\}_{n\in \mathbb{N}},
	\end{align*}

	It's easy to see that $C$ and $C^{'}$ are bounded, linear and invertible operators.\\
	Furthermore, 
	\begin{align*}
	\sum_{k\in \mathbb{N}}\langle \Lambda_{k}C(\{a_{n}\}_{n\in \mathbb{N}}),\Lambda_{k}C^{'}(\{a_{n}\}_{n\in \mathbb{N}}) \rangle_{\mathcal{A}}&=\sum_{k\in \mathbb{N}}\langle \Lambda_{k}(\{\alpha a_{n}\}_{n\in \mathbb{N}}),\Lambda_{k}(\{\beta a_{n}\}_{n\in \mathbb{N}}) \rangle_{\mathcal{A}}\\
	&=\sum_{k\in \mathbb{N}}\langle \bigg\{\frac{\alpha a_{n}\delta_{n}^{2k+1}}{\sqrt{2k+1}}\bigg\}_{n\in \mathbb{N}},\bigg\{\frac{\beta a_{n}\delta_{n}^{2k+1}}{\sqrt{2k+1}}\bigg\}_{n\in \mathbb{N}}\rangle_{\mathcal{A}}\\
	&=\sum_{k\in \mathbb{N}} \alpha \beta \bigg\{\frac{| a_{n}|^{2}\delta_{n}^{2k+1}}{(2k+1)n}\bigg
	\}_{n\in \mathbb{N}}\\
	&=\alpha \beta \bigg\{\frac{|a_{2k+1}|^{2}}{(2k+1)^{2}}\bigg\}_{k\in \mathbb{N}}\\
		\end{align*}
		Then the family $\{\Lambda_{k}\}_{k\in \mathbb{N}}$ is a $\ast$-$(C,C^{'})$-controlled Bessel sequence, but is not a $(C,C^{'})$-controlled operator frame.\\
		Indeed,
		\begin{align*}
	\sum_{k\in \mathbb{N}}\langle \Lambda_{k}C\{a_{n}\}_{n\in \mathbb{N}},\Lambda_{k}C^{'}\{a_{n}\}_{n\in \mathbb{N}}\rangle_{\mathbb{A}}&\leq \alpha \beta \bigg\{\frac{|a_{2k+1}|^{2}}{(2k+1)}\bigg\}_{k\in \mathbb{N}}\\
	&\leq \alpha \beta \bigg\{\frac{|a_{k}|^{2}}{k}\bigg\}_{k\in \mathbb{N}}\\
	&=\alpha\beta \langle \{a_{n}\}_{n\in \mathbb{N}},\{a_{n}\}_{n\in \mathbb{N}}\rangle_{\mathbb{A}}\\
	&=\sqrt{\alpha\beta}Id_{\mathbb{A}}\langle \{a_{n}\}_{n\in \mathbb{N}},\{a_{n}\}_{n\in \mathbb{N}}\rangle_{\mathbb{A}}\sqrt{\alpha\beta}Id_{\mathcal{A}}
		\end{align*}
		But if $\{a_{n}\}_{n\in \mathbb{N}}\in \mathcal{H}$ and $a_{k+1}=0 \quad \forall k\in \mathbb{N}$, then
		\begin{equation*}
	\langle \Lambda_{k}\{a_{n}\}_{n\in \mathbb{N}},\Lambda_{k}\{a_{n}\}_{n\in \mathbb{N}}\rangle_{\mathcal{A}}=\{0\}
		\end{equation*}
	So, we consider the operator $K$ defined by,
	\begin{align*}
	K:\mathcal{H}&\longrightarrow \mathcal{H}\\
	\{a_{n}\}_{n\in \mathbb{N}}&\longrightarrow \{a_{2k+1}\}_{k\in \mathbb{N}}
\end{align*}
and $A=\sqrt{\alpha\beta}\bigg\{\frac{1}{\sqrt{k}}\bigg\}_{k\in \mathbb{N}^{\ast}}$, then,
\begin{align*}
A\langle K^{\ast} \{a_{n}\}_{n\in \mathbb{N}},K^{\ast} \{a_{n}\}_{n\in \mathbb{N}}\rangle_{\mathcal{A}}A^{\ast}=&\sqrt{\alpha\beta}\bigg\{\frac{1}{\sqrt{k}}\bigg\}_{k\in \mathbb{N}^{\ast}}\bigg\{\frac{|a_{2k+1}|^{2}}{2k+1}\bigg\}_{k\in \mathbb{N}}\sqrt{\alpha\beta}\bigg\{\frac{1}{\sqrt{k}}\bigg\}_{k\in \mathbb{N}^{\ast}}\\
&=\alpha\beta\bigg\{\frac{|a_{2k+1}|^{2}}{(2k+1)^{2}}\bigg\}_{k\in \mathbb{N}}\\
&=\sum_{k\in \mathbb{N}}\langle \Lambda_{k}C\{a_{n}\}_{n\in \mathbb{N}},\Lambda_{k}C^{'}\{a_{n}\}_{n\in \mathbb{N}}\rangle_{\mathbb{A}}
\end{align*}
Wich proof that $\{\Lambda_{k}\}_{k\in \mathbb{N}}$ is a tight $(C,C^{'})$-$\ast$-controlled operator frame.
\end{example}
\begin{proposition}
	Every $(C,C^{'})$-controlled $\ast$-operator frame for $End_{\mathcal{A}}^{\ast}(\mathcal{H})$ is a $(C,C^{'})$-controlled $\ast$-K-operator frame for $End_{\mathcal{A}}^{\ast}(\mathcal{H})$.
\end{proposition}
\begin{proof}
	For any $K \in End_{\mathcal{A}}^{\ast}(\mathcal{H}) $, we have,
	$$\langle K^{\ast}x,K^{\ast}x\rangle_{\mathcal{A}} \leq \|K\|^2 \langle x,x\rangle_{\mathcal{A}}.$$
	Let $\{T_{i}\}_{i\in I}$ be a $(C,C^{'})$-controlled $\ast$-operator frame for $End_{\mathcal{A}}^{\ast}(\mathcal{H})$ with bounds A and B.\\
	Then, 
	$$A\langle x,x\rangle_{\mathcal{A}} A^{\ast}\leq\sum_{i\in I}\langle T_{i}Cx,T_{i}C^{'}x\rangle_{\mathcal{A}}\leq B\langle x,x\rangle_{\mathcal{A}}B^{\ast},  x\in\mathcal{H}.$$
	Hence, 
	$$A\|K\|^{-2} \langle K^{\ast}x,K^{\ast}x\rangle_{\mathcal{A}} A^{\ast}\leq\sum_{i\in I}\langle T_{i}Cx,T_{i}C^{'}x\rangle_{\mathcal{A}}\leq B\langle x,x\rangle_{\mathcal{A}}B^{\ast},  x\in\mathcal{H}.$$
	Thus 
	$$A\|K\|^{-1} \langle K^{\ast}x,K^{\ast}x\rangle_{\mathcal{A}} (A\|K\|^{-1})^{\ast}\leq\sum_{i\in I}\langle T_{i}Cx,T_{i}C^{'}x\rangle_{\mathcal{A}}\leq B\langle x,x\rangle_{\mathcal{A}}B^{\ast},  x\in\mathcal{H}.$$
	Therefore, $\{T_{i}\}_{i\in I}$  is a $(C,C^{'})$-controlled $\ast$ - K-operator frame for $End_{\mathcal{A}}^{\ast}(\mathcal{H})$ with bounds $A\|K\|^{-1}$ and B.
\end{proof}
\begin{proposition}
	Let $\{T_{i}\}_{i\in I}$ be a $(C,C^{'})$-controlled $\ast$-K-operator frame for $End_{\mathcal{A}}^{\ast}(\mathcal{H})$. If K is surjective then $\{T_{i}\}_{i\in I}$ is a $(C,C^{'})$-controlled $\ast$-operator frame for $End_{\mathcal{A}}^{\ast}(\mathcal{H})$.
\end{proposition}
\begin{proof}
	Suppose that K is surjective, from lemma \ref{111}  there exists $m>0$ such that
	
\begin{equation} \label{do11}
      \langle K^{\ast}x,K^{\ast}x\rangle_\mathcal{A}\geq m\langle x,x\rangle_\mathcal{A} , x\in\mathcal{H}
\end{equation}
   Let $\{T_{i}\}_{i\in I}$ be a $(C,C^{'})$-controlled $\ast$-K-operator frame for $End_{\mathcal{A}}^{\ast}(\mathcal{H})$ with bounds A and B, then, 
\begin{equation} \label{do12}
    A\langle K^{\ast}x,K^{\ast}x\rangle_{\mathcal{A}}A^{\ast} \leq\sum_{i\in I}\langle T_{i}Cx,T_{i}C^{'}x\rangle_{\mathcal{A}}\leq B\langle x,x\rangle_{\mathcal{A}}B^{\ast},  x\in\mathcal{H}.
\end{equation}
    Using (\ref{do11})	and (\ref{do12}), we have
    $$A m\langle x,x\rangle_{\mathcal{A}} A^{\ast}\leq\sum_{i\in I}\langle T_{i}Cx,T_{i}C^{'}x\rangle_{\mathcal{A}}\leq B\langle x,x\rangle_{\mathcal{A}}B^{\ast},  x\in\mathcal{H}.$$
    Then
    $$A \sqrt{m}\langle x,x\rangle_{\mathcal{A}} (A\sqrt{m})^{\ast}\leq\sum_{i\in I}\langle T_{i}Cx,T_{i}C^{'}x\rangle_{\mathcal{A}}\leq B\langle x,x\rangle_{\mathcal{A}}B^{\ast},  x\in\mathcal{H}.$$
    Therefore  $\{T_{i}\}_{i\in I}$ is a $(C,C^{'})$-controlled $\ast$-operator frame for $End_{\mathcal{A}}^{\ast}(\mathcal{H})$.
\end{proof}
\begin{proposition}
	Let  $C,C^{'} \in GL^{+}(\mathcal{H})$ and $\{T_{i}\}_{i\in I}$ be a $\ast$- K-operator frame for $End_{\mathcal{A}}^{\ast}(\mathcal{H})$. Assume that $C$ and $C'$ commute with $T_i$ for each $i\in I$ and commute with $K$. Then $\{T_{i}\}_{i\in I}$ is a $(C,C^{'})$-controlled $\ast$-K-operator frame for $End_{\mathcal{A}}^{\ast}(\mathcal{H})$.
	
\end{proposition}
\begin{proof}
	Let $\{T_{i}\}_{i\in I}$ be a $\ast$- K-operator frame for $End_{\mathcal{A}}^{\ast}(\mathcal{H})$.\\
	Then there exist two strictly non zero elements $A$ and $B$ in $\mathcal{A}$ such that 
	\begin{equation}\label{eq33}
	A\langle K^{\ast}x,K^{\ast}x\rangle_{\mathcal{A}}A^{\ast}\leq\sum_{i\in I}\langle T_{i}x,T_{i}x\rangle_{\mathcal{A}}\leq B\langle x,x\rangle_{\mathcal{A}}B^{\ast},  x\in\mathcal{H}.
	\end{equation} 
On one hand, we have, 
\begin{align*}
    \sum_{i\in I}\langle T_{i}Cx,T_{i}C'x\rangle_{\mathcal{A}}&= \sum_{i\in I}\langle T_{i}(CC')^{\frac{1}{2}}x,T_{i}(CC')^{\frac{1}{2}}x\rangle_{\mathcal{A}}\\ 
    &\leq B\langle (CC')^{\frac{1}{2}} x,(CC')^{\frac{1}{2}} x\rangle_{\mathcal{A}}\\
    &\leq B \|(CC')^{\frac{1}{2}}\|^2 \langle x,x\rangle_{\mathcal{A}},
\end{align*}
   So, 
\begin{equation} \label{do23}
     \sum_{i\in I}\langle T_{i}Cx,T_{i}C'x\rangle_{\mathcal{A}}\leq B \|(CC')^{\frac{1}{2}}\|^2 \langle x,x\rangle_{\mathcal{A}}.
\end{equation}
   Also, we have, 
\begin{align*}
    \sum_{i\in I}\langle T_{i}(CC')^{\frac{1}{2}}x,T_{i}(CC')^{\frac{1}{2}}x\rangle_{\mathcal{A}}
    & \geq 	A\langle K^{\ast}(CC')^{\frac{1}{2}}x,K^{\ast}(CC')^{\frac{1}{2}}x\rangle_{\mathcal{A}}A^{\ast}\\ 
    &\geq 	A\langle (CC')^{\frac{1}{2}} K^{\ast}x,(CC')^{\frac{1}{2}}K^{\ast}x\rangle_{\mathcal{A}}A^{\ast}.
\end{align*} 
    Since $(CC')^{\frac{1}{2}}$ is a surjecive operator, then there exists $m>0$ such that,
\begin{equation} \label{do24}
    \langle (CC')^{\frac{1}{2}} K^{\ast}x,(CC')^{\frac{1}{2}}K^{\ast}x\rangle_{\mathcal{A}} \geq m \langle  K^{\ast}x,K^{\ast}x\rangle_{\mathcal{A}}.
\end{equation} 
    Which give,
    $$ A m \langle  K^{\ast}x,K^{\ast}x\rangle_{\mathcal{A}}A^{\ast} \leq \sum_{i\in I}\langle T_{i}Cx,T_{i}C'x\rangle_{\mathcal{A}}\leq B \|(CC')^{\frac{1}{2}}\|^2 \langle x,x\rangle_{\mathcal{A}}B^{\ast}. $$ 
    Hence
    $$ A \sqrt{m} \langle  K^{\ast}x,K^{\ast}x\rangle_{\mathcal{A}}(A\sqrt{m})^{\ast} \leq \sum_{i\in I}\langle T_{i}Cx,T_{i}C'x\rangle_{\mathcal{A}}\leq B \|(CC')^{\frac{1}{2}}\| \langle x,x\rangle_{\mathcal{A}}(B\|(CC')^{\frac{1}{2}}\|)^{\ast}. $$ 
    Therefore $\{T_{i}\}_{i\in I}$ is a $(C,C^{'})$-controlled ${\ast}$-K-operator frame for $End_{\mathcal{A}}^{\ast}(\mathcal{H})$. 
	
\end{proof}
 Let $\{T_i\}_{i\in I}$ be a $(C,C')-$controlled Bessel  ${\ast}$-operator frame for $End_\mathcal{A}^\ast (\mathcal{H})$. We assume that $C$ and $C'$ commute between them and commute with $T^{\ast}_iT_i$ for each $i\in I$.\\
    We define the operator $T_{(C,C')}$ :$ \mathcal{H} \rightarrow l^2(\mathcal{H})$
    given by $$T_{(C,C')}x=\{T_i(CC')^\frac{1}{2}x\}_{i\in I}.$$\\
    $T_{(C,C')}$ is called the analysis operator.
    The adjoint operator for $T$ is defined by $T^{\ast}_{(C,C^{'})} : l^2(\mathcal{H})  \rightarrow \mathcal{H}$ given by, $$T^{\ast}_{(C,C^{'})}(\{a_i\}_{i \in I})=\sum_{i\in I} (CC')^\frac{1}{2}T_i^{\ast}a_i$$
    is called the synthesis operator.\\
    We define the frame operator of the $(C,C^{'})$-controlled Bessel  ${\ast}$-operator frame by:
 \begin{align*}
    S_{(C,C^{'})}:&\mathcal{H}\longrightarrow \mathcal{H}\\
    &x\longrightarrow S_{(C,C^{'})}x=T_{(C,C^{'})}T^{\ast}_{(C,C^{'})}x=\sum_{i\in I}  C'T_i^{\ast}T_iCx.
\end{align*}
    It's clear to see that $S_{(C,C^{'})}$ is positive, bounded and selfadjoint.
     	
\begin{theorem}
	Let $K,Q \in End_\mathcal{A}^\ast (\mathcal{H})$ and $\{T_i\}_{i\in I}\subset End_\mathcal{A}^\ast (\mathcal{H})$ be a $(C,C')-$controlled  ${\ast}$-$K$-operator frame for $End_\mathcal{A}^\ast (\mathcal{H})$ with frame operator $S_{(C,C')}$. Suppose that $Q$ commute with $C$ , $C'$ and commute with $K$. Then $\{T_iQ\}_{i\in I}$ is a $(C,C')-$controlled  ${\ast}$- $(Q^{\ast}K)$-operator frame for $End_\mathcal{A}^\ast (\mathcal{H})$ with frame operator $S=Q^{\ast}S_{(C,C')}Q$ .
	
\end{theorem}
\begin{proof}
	Suppose that $\{T_i\}_{i\in I}$  is a $(C,C')-$controlled  ${\ast}$- K-operator frame with frame bounds A and B. Then, $$A\langle K^{\ast}x,K^{\ast}x\rangle_{\mathcal{A}} A^{\ast}\leq\sum_{i\in I}\langle T_{i}Cx,T_{i}C^{'}x\rangle_{\mathcal{A}}\leq B\langle x,x\rangle_{\mathcal{A}}B^{\ast}.$$
	Hence, 
	$$A\langle K^{\ast}Qx,K^{\ast}Qx\rangle_{\mathcal{A}}A^{\ast} \leq\sum_{i\in I}\langle T_{i}CQx,T_{i}C^{'}Qx\rangle_{\mathcal{A}}\leq B\langle Qx,Qx\rangle_{\mathcal{A}}B^{\ast}.$$
	So, 
	$$A\langle (Q^{\ast}K)^{\ast}x,(Q^{\ast}K)^{\ast}x\rangle_{\mathcal{A}}A^{\ast} \leq\sum_{i\in I}\langle T_{i}QCx,T_{i}QC^{'}x\rangle_{\mathcal{A}}\leq B \|Q\|^2\langle x,x\rangle_{\mathcal{A}}B^{\ast}.$$
	Then
	$$A\langle (Q^{\ast}K)^{\ast}x,(Q^{\ast}K)^{\ast}x\rangle_{\mathcal{A}}A^{\ast} \leq\sum_{i\in I}\langle T_{i}QCx,T_{i}QC^{'}x\rangle_{\mathcal{A}}\leq B \|Q\|\langle x,x\rangle_{\mathcal{A}}(B\|Q\|)^{\ast}.$$
	Therefore $\{T_iQ\}_{i\in I}$ is a $(C,C')$-controlled  ${\ast}$-$(Q^{\ast}K)$-operator frame for $End_\mathcal{A}^\ast (\mathcal{H})$ with bounds A and $B \|Q\|$.\\
	Furthermore,
\begin{align*}
   S= \sum_{i\in I}  C'(T_iQ)^{\ast}T_iQCx&=\sum_{i\in I}  C'Q^{\ast}T_i^{\ast}T_iCQx=Q^{\ast}\sum_{i\in I}  C'T_i^{\ast}T_iCQx\\
    &=Q^{\ast}S_{(C,C')}Qx
\end{align*}
	\end{proof}

\begin{theorem}
	Let $\{T_i\}_{i\in I}$ be a $(C,C')-$controlled  $\ast$-K-operator frame for $End_\mathcal{A}^\ast (\mathcal{H})$ with best frame bounds A and B. If  $Q : \mathcal{H}  \rightarrow \mathcal{H}$
	is an adjointable and invertible operator such that $Q$ and $Q^{-1}$ commutes with $C$ and $C^{'}$ for each $i\in I$ and $Q^{-1}$ commute with $K^{\ast}$, then $\{T_iQ\}_{i\in I}$ is a $(C,C')-$controlled  $\ast$-K-operator frame for $End_\mathcal{A}^\ast (\mathcal{H})$ with best frame bounds M and N satisfying the inequalities,

\begin{equation}\label{eq10}
A\|Q^{-1}\|^{-1}\leq M \leq A\|Q\| \qquad and \qquad A\|Q^{-1}\|^{-1}\leq N \leq B\|Q\|.
\end{equation}

\end{theorem}
\begin{proof}
	Let $\{T_i\}_{i\in I}$ be a $(C,C')-$controlled  $\ast$-K-operator frame for $End_\mathcal{A}^\ast (\mathcal{H})$ with best frame bounds A and B.
	
	One one hand, we have
	\begin{align*}
\sum_{i\in I}\langle T_{i}CQx,T_{i}C^{'}Qx\rangle_{\mathcal{A}}&\leq B\langle Qx,Qx\rangle_{\mathcal{A}}B^{\ast}\\
&\leq B \|Q\|^2\langle x,x\rangle_{\mathcal{A}}B^{\ast}\\
&\leq B \|Q\|\langle x,x\rangle_{\mathcal{A}}(B\|Q\|)^{\ast}.
	\end{align*}
	 
	One the other hand, we have,
\begin{align*}
    	A\langle K^{\ast}x,K^{\ast}x\rangle_{\mathcal{A}}A^{\ast}&=A\langle K^{\ast}Q^{-1}Qx,K^{\ast}Q^{-1}Qx\rangle_{\mathcal{A}}A^{\ast}\\
    	&=A\langle Q^{-1} K^{\ast}Qx,Q^{-1} K^{\ast}Qx\rangle_{\mathcal{A}}A^{\ast}\\
    	&\leq A\| Q^{-1}\|^2\langle K^{\ast}Qx, K^{\ast}Qx\rangle_{\mathcal{A}}A^{\ast}\\
    	&\leq \| Q^{-1}\|^2 \sum_{i\in I}\langle T_{i}CQx,T_{i}C^{'}Qx\rangle_{\mathcal{A}}\\
    	&= \| Q^{-1}\|^2 \sum_{i\in I}\langle T_{i}QCx,T_{i}QC^{'}x\rangle_{\mathcal{A}}.
\end{align*}
       Hence, 
       $$A\| Q^{-1}\|^{-2}\langle K^{\ast}x,K^{\ast}x\rangle_{\mathcal{A}}A^{\ast}\leq \sum_{i\in I}\langle T_{i}QCx,T_{i}QC^{'}x\rangle_{\mathcal{A}}\leq B \|Q\|\langle x,x\rangle_{\mathcal{A}}(B\|Q\|)^{\ast}.$$
       Thus
        $$A\| Q^{-1}\|^{-1}\langle K^{\ast}x,K^{\ast}x\rangle_{\mathcal{A}}(A\| Q^{-1}\|^{-1})^{\ast}\leq \sum_{i\in I}\langle T_{i}QCx,T_{i}QC^{'}x\rangle_{\mathcal{A}}\leq B \|Q\|\langle x,x\rangle_{\mathcal{A}}(B\|Q\|)^{\ast}.$$
       Therefore, $\{T_iQ\}_{i\in I}$ is a $(C,C')-$controlled $\ast$-K-operator frame for $End_\mathcal{A}^\ast (\mathcal{H})$ with bounds $A\| Q^{-1}\|^{-1}$ and $ B \|Q\|$.\\
       Now, let M and N be the best bounds of the $(C,C')-$controlled  $\ast$- K-operator frame $\{T_iQ\}_{i\in I}$. Then,
       
\begin{equation}\label{eq8}
       A\|Q^{-1}\|^{-1} \leq M\;\;\;\;\;\; and \;\;\;\;\;\;N\leq B \|Q\| .
\end{equation}
    
       Also, $\{T_{i}Q\}_{i\in I}$ is a $(C,C')-$controlled $\ast$-$K$-operator frame for $End_{\mathcal{A}}^{\ast}(\mathcal{H})$ with frame bounds $M$ and $N$. 
       
       Since,
       $$\langle K^{\ast}x,K^{\ast}x\rangle_{\mathcal{A}} =\langle QQ^{-1}K^{\ast}x,QQ^{-1}K^{\ast}x\rangle_{\mathcal{A}} \leq\|Q\|^{2}\langle K^{\ast}Q^{-1}x,K^{\ast}Q^{-1}x\rangle_{\mathcal{A}} , x\in\mathcal{H}.$$
      So,
      \begin{align*}
      M\|Q\|^{-2}\langle K^{\ast}x,K^{\ast}x\rangle_{\mathcal{A}} M^{\ast} &\leq M\langle K^{\ast}Q^{-1}x,K^{\ast}Q^{-1}x\rangle_{\mathcal{A}}M^{\ast}  \\
      &\leq\sum_{i\in I}\langle T_{i}QCQ^{-1}x,T_{i}QC'Q^{-1}x\rangle_{\mathcal{A}} \\
      &=\sum_{i\in I}\langle T_{i}QQ^{-1}Cx,T_{i}QQ^{-1}C'x\rangle_{\mathcal{A}} \\
      &=\sum_{i\in I}\langle T_{i}Cx,T_{i}C^{'}x\rangle_{\mathcal{A}} \\
      &\leq N\|Q^{-1}\|^{2}\langle x,x\rangle_{\mathcal{A}}\\
      &\leq N\|Q^{-1}\|\langle x,x\rangle_{\mathcal{A}}(N\|Q^{-1}\|)^{\ast}.
       \end{align*}
      Since A and B are the best bounds of $(C,C')-$controlled  K-operator frame $\{T_i\}_{i\in I}$, we have 
      \begin{equation}\label{eq9}
      M\|Q\|^{-1}\leq A \;\;\;\;\;\; and \;\;\;\;\;\; B\leq N\|Q^{-1}\|.
      \end{equation}
      Therfore the inequality (\ref{eq10}) follows from (\ref{eq9})   and (\ref{eq8}).
\end{proof}

\begin{theorem}
	Let $(\mathcal{H}, \mathcal{A}, \langle .,.\rangle_{\mathcal{A}})$ and $(\mathcal{H}, \mathcal{B}, \langle .,.\rangle_{\mathcal{B}})$ be two Hilbert $C^{\ast}$-modules and let $\varphi$ $\mathcal{A} \longrightarrow \mathcal{B}$ be a $\ast$-homomorphisme and $\theta$ be a map on $\mathcal{H}$ such that $\langle \theta x,\theta y\rangle_{\mathcal{B}}=\varphi( \langle x,y\rangle_{\mathcal{A}}  $ for all $x, y \in \mathcal{H} $. Suppose $\{T_{i}\}_{i\in I} \subset End_{\mathcal{A}}^{\ast}(\mathcal{H})$ is a $(C,C^{'})$-controlled $\ast$-$K$-operator frame for $(\mathcal{H}, \mathcal{A}, \langle .,.\rangle_{\mathcal{A}})$ with frame operator $S_{\mathcal{A}}$ and lower and upper bounds A and B respectively. If $\theta$ is surjective such that  $\theta T_i=T_i \theta$ for each $i \in I$,  $\theta C= C \theta$ , $\theta C^{'}= C^{'} \theta$ and $\theta K^{\ast}= K^{\ast} \theta$, then $\{T_{i}\}_{i\in I}$ is a $(C,C^{'})$-controlled $\ast$-$K$-operator frame for $(\mathcal{H}, \mathcal{B}, \langle .,.\rangle_{\mathcal{B}})$ with frame operator $S_{\mathcal{B}}$  and lower and upper  bounds $\varphi (A)$,  $\varphi (B)$ respectively and  $\langle S_{\mathcal{B}}\theta x, \theta y\rangle_{\mathcal{B}}= \varphi ( \langle S_{\mathcal{A}} x,  y\rangle_{\mathcal{A}} $.
	
\end{theorem}
\begin{proof}
	Since $\theta $ is surjective, then for every $y \in \mathcal{H}$ there exists $x \in \mathcal{H}$ such that $\theta x= y$. Using  the definition of $(C,C^{'})$-controlled $\ast$- K-operator frame we have,
	$$ A\langle K^{\ast}x,K^{\ast}x\rangle _{\mathcal{A}} A^{\ast} \leq\sum_{i\in I}\langle T_{i}Cx,T_{i}C^{'}x\rangle\leq B\langle x,x\rangle_{\mathcal{A}} B^{\ast},  x\in\mathcal{H}.$$
	We have for all $x\in \mathcal{H}$, 
	$$\varphi (A\langle K^{\ast} x,K^{\ast}x\rangle _{\mathcal{A}} A^{\ast} )\leq \varphi (\sum_{i\in I}\langle T_{i}Cx,T_{i}C^{'}x\rangle_{\mathcal{A}})\leq \varphi (B\langle x,x\rangle_{\mathcal{A}} B^{\ast}).$$
	From the definition of $\ast$-homomorphisme we have 
	$$\varphi (A) \varphi (\langle K^{\ast} x,K^{\ast}x\rangle _{\mathcal{A}}) \varphi( A^{\ast} )\leq \varphi (\sum_{i\in I}\langle T_{i}Cx,T_{i}C^{'}x\rangle_{\mathcal{A}})\leq \varphi (B) \varphi (\langle x,x\rangle_{\mathcal{A}}) \varphi ( B^{\ast}).$$
	Using the relation betwen $\theta$ and $\varphi $ we get 
	$$\varphi (A) \langle \theta K^{\ast}x,\theta K^{\ast}x\rangle _{\mathcal{B}} (\varphi( A))^{\ast} \leq \sum_{i\in I}\langle \theta T_{i}Cx,\theta T_{i}C^{'}x\rangle_{\mathcal{B}}\leq \varphi (B)  \langle \theta x, \theta x\rangle_{\mathcal{B}}) (\varphi ( B))^{\ast}.$$
	Since $\theta T_i=T_i \theta$  ,  $\theta C= C \theta$, $\theta C^{'}= C^{'} \theta$ and $\theta K^{\ast}= K^{\ast} \theta$ we have 
	$$\varphi (A) \langle K^{\ast}\theta x,K^{\ast}\theta x\rangle _{\mathcal{B}} (\varphi( A))^{\ast} \leq \sum_{i\in I}\langle  T_{i} C\theta x, T_{i} C^{'}\theta x\rangle_{\mathcal{B}}\leq \varphi (B)  \langle \theta x, \theta x\rangle_{\mathcal{B}}) (\varphi ( B))^{\ast}.$$
	Therefore 
	$$\varphi (A) \langle K^{\ast} y,K^{\ast}y \rangle _{\mathcal{B}} (\varphi( A))^{\ast} \leq \sum_{i\in I}\langle  T_{i} C y, T_{i} C^{'}y\rangle_{\mathcal{B}}\leq \varphi (B)  \langle y, y\rangle_{\mathcal{B}}) (\varphi ( B))^{\ast},  y\in\mathcal{H}.$$
	This implies that $\{T_{i}\}_{i\in I}$ is a $(C,C^{'})$-controlled $\ast$-K-operator frame for $(\mathcal{H}, \mathcal{B}, \langle .,.\rangle_{\mathcal{B}})$ with bounds $\varphi (A)$ and $\varphi (B)$. 
	Moreover we have 
	\begin{align*}
	\varphi ( \langle S_{\mathcal{A}} x,  y\rangle_{\mathcal{A}}&= \varphi ( \langle \sum_{i\in I}  T_{i}Cx, T_{i}C^{'}y \rangle_{\mathcal{A}})\\
	&=  \sum_{i\in I}\varphi ( \langle T_{i}Cx, T_{i}C^{'}y \rangle_{\mathcal{A}})\\
	&=\sum_{i\in I} \langle \theta T_{i}Cx, \theta T_{i}C^{'}y \rangle_{\mathcal{B}}\\
	&= \sum_{i\in I} \langle  T_{i}C\theta x,  T_{i}C^{'}\theta y \rangle_{\mathcal{B}}\\
	&= \langle \sum_{i\in I}C^{'} T_{i}^{\ast} T_{i}C\theta x,  \theta y \rangle_{\mathcal{B}}\\
	&=\langle S_{\mathcal{B}} \theta x,  \theta y\rangle_{\mathcal{A}}).
	\end{align*}
	Which completes the proof.
\end{proof}
\section{\textbf{Operators preserving controlled $\ast$-$ K$-operator frame}}

\begin{proposition}
	Let $K,U \in End_\mathcal{A}^\ast (\mathcal{H})$ such that $R(U)\subset R(K)$. If  $\{T_i\}_{i\in I}$ is a $(C,C')-$controlled  $\ast$-K-operator frame for $End_\mathcal{A}^\ast (\mathcal{H})$, then $\{T_i\}_{i\in I}$ is a $(C,C')-$controlled  $\ast$-U-operator frame for $End_\mathcal{A}^\ast (\mathcal{H})$
	
\end{proposition}
\begin{proof}
	Assume that $\{T_i\}_{i\in I}$ is a $(C,C')-$controlled  $\ast$-K-operator with bounds A and B, then 
\begin{equation}\label{haj1}
    A\langle K^{\ast}x,K^{\ast}x\rangle_{\mathcal{A}}A^{\ast} \leq\sum_{i\in I}\langle T_{i}Cx,T_{i}C^{'}x\rangle_{\mathcal{A}}\leq B \langle x,x\rangle_{\mathcal{A}}B^{\ast},  x\in\mathcal{H}.
\end{equation}

Since $R(U)\subset R(K)$ then from lemma (\ref{3}), there exists $\lambda >0 $ such that $UU^{\ast}\leq \lambda KK^{\ast}$. Using (\ref{haj1}), we have 
	$$\frac{A}{\sqrt{\lambda }}\langle U^{\ast}x,U^{\ast}x\rangle_{\mathcal{A}}(\frac{A}{\sqrt{\lambda }})^{\ast} \leq\sum_{i\in I}\langle T_{i}Cx,T_{i}C^{'}x\rangle_{\mathcal{A}}\leq B \langle x,x\rangle_{\mathcal{A}}B^{\ast},  x\in\mathcal{H}.$$
	Therefore $\{T_i\}_{i\in I}$ is a $(C,C')-$controlled  $\ast$-U-operator frame for $End_\mathcal{A}^\ast (\mathcal{H})$.
\end{proof}
\begin{theorem}
	Let $K \in End_\mathcal{A}^\ast (\mathcal{H})$ with a dense range. Let $\{T_i\}_{i\in I}$ be a $(C,C')-$controlled  $\ast$-K-operator and $U \in End_\mathcal{A}^\ast (\mathcal{H})$  have closed range and commute with C and $C^{'}$. If $\{T_iU\}_{i\in I}$ and $\{T_iU^\ast\}_{i\in I}$ are $(C,C')-$controlled  $\ast$-K-operator frame for $End_\mathcal{A}^\ast (\mathcal{H})$ then U is invertible.

\end{theorem}
\begin{proof}
	Suppose that $\{T_iU\}_{i\in I}$ is a $(C,C')-$controlled  $\ast$-K-operator frame with bounds $A_1$ and $B_1$, then 
\begin{equation}\label{haj2}
	A_1\langle K^{\ast}x,K^{\ast}x\rangle_{\mathcal{A}}A_1^{\ast} \leq\sum_{i\in I}\langle T_{i}UCx,T_{i}UC^{'}x\rangle_{\mathcal{A}}\leq B_1 \langle x,x\rangle_{\mathcal{A}}B_1^{\ast},  x\in\mathcal{H}.
\end{equation}
    Since K have a dense range then $K^{\ast}$ is injective.\\ Hence from (\ref{haj2}), $N(U) \subset N(K^{\ast})$. Then  $U$ is injective.\\ Moreover $R(U^{\ast })=N(U)^{\perp}=\mathcal{H}$. Therefore U is surjective. \\
    Now assume that $\{T_iU^{\ast}\}_{i\in I}$ is a $(C,C')-$controlled  $\ast$-K-operator frame with bounds $A_2$ and $B_2$, then 
\begin{equation}\label{haj3}
    A_2\langle K^{\ast}x,K^{\ast}x\rangle_{\mathcal{A}}A_2^{\ast} \leq\sum_{i\in I}\langle T_{i}U^{\ast}Cx,T_{i}U^{\ast}C^{'}x\rangle_{\mathcal{A}}\leq B_2 \langle x,x\rangle_{\mathcal{A}}B_2^{\ast},  x\in\mathcal{H}.
\end{equation}
    Hence U is injective, since $N(U^{\ast})$ $\subset N(K^{\ast})$. Thus U is invertible. 
\end{proof}

ACKNOWLEDGMENTS\\
Sincere thanks goes to the valuable comments of the referees and the editors that give a step forward to the main file of the manuscript. Also the authors would like to thanks the reviewers for their valuable comments.

\bibliographystyle{amsplain}

\vspace{0.1in}

\end{document}